\documentclass[a4paper,12pt]{article}
\usepackage{color}
\usepackage{latexsym}
\usepackage{amsmath}
\usepackage{amssymb}
\usepackage{enumerate}
\usepackage{pict2e}
\usepackage{theorem}
\newtheorem{theorem}{Theorem}[section]
\newtheorem{proposition}[theorem]{Proposition}
\newtheorem{lemma}[theorem]{Lemma}

\theorembodyfont{\rmfamily}
\newtheorem{proof}{\textmd{\textit{Proof.}}}

\newtheorem{remark}[theorem]{Remark}

\makeatletter

\@addtoreset{equation}{section}
\makeatother

\newcommand{\qedd}{\hfill \Box}

\newcommand{\B}{\ensuremath{\mathbb{B}}}

\newcommand{\R}{\ensuremath{\mathbb{R}}}
\newcommand{\Sph}{\ensuremath{\mathbb{S}}}

\usepackage{subcaption}
\usepackage{color}
\usepackage{latexsym}
\usepackage{amsmath}
\usepackage{amssymb}
\usepackage{enumerate}
\usepackage{pict2e}
\usepackage{graphicx}
\usepackage{float}
\title{The geometry on the slope of a mountain}
\author{P. Chansri, P. Chansangiam and S. V. Sabau 
\footnote{Corresponding author}
}
\date{}
\begin{document}

\maketitle
\begin{abstract}
The geometry on a slope of a mountain is the geometry of a Finsler metric, called here the {\it slope metric}. We study the existence of globally defined slope metrics on surfaces of revolution as well as the geodesic's behavior. A comparison between Finslerian and Riemannian areas of a bounded region is also studied.
\end{abstract}
\let\thefootnote\relax\footnote{Miskolc Mathematical Notes, Vol. 21 (2020), No. 2, pp. 747–762.}
\section{Introduction}

\quad Finsler manifolds, that is $n$-dimensional smooth manifolds endowed with Finsler metrics, are natural generalization of the well-known Riemannian manifolds. The main difference is that the metric itself and all Finsler geometric quantities depend not only on the point $x \in M$ of the manifold, but also on the direction $y \in T_xM$, where $(x,y)$ are the canonical coordinates of the tangent bundle $TM$. This directional dependence reveals many  hidden geometrical features that are usually obscured by the quadratic form in the $y$-variable of a Riemannian metric. On the other hand, most of the geometrical properties of Finsler spaces are highly nonlinear, this is the case with the non-linear connection or the parallel displacement, making most of the traditional Riemannian methods unapplicable. 

It is well-known that one of the most important problems in differential geometry and calculus of variations is the time minimizing travel between two points on a Riemannian or Finsler manifold. The problem of finding these time minimizing paths goes back to Caratheodory (\cite{C}) and Finsler himself and can be directly related to the Hilbert's fourth problem (see \cite{AIM} for details).

An important insight in to the problem is due to Shen (\cite{S}) who related the Zermelo's navigation problem to the geometry of Randers metrics. Indeed, it is now clear that the time minimizing travel paths on a Riemannian manifold $(M,h)$ under the influence a mild wind $W\in TM$, $||W||_h<1$, are exactly the geodesics of a Randers metric $F=\alpha+\beta$ uniquely determined by the navigation data $(h,W)$ (see  \cite{BRS} for details).

Moreover, a singular solution of the Zermelo's navigation problem can be found in the case $||W||_h=1$, namely the geodesics of a Kropina metric (\cite{YS}). The Randers metrics $F=\alpha+\beta$ and the Kropina metrics $F=\frac{\alpha^2}{\beta}$  belong to a larger class of Finsler metrics called $(\alpha, \beta)$- metrics since  they are obtained by deformations of a Riemannian metric by means of a linear 1-form $\beta=b_i(x)y^i$ on $TM$. The common characteristic is that they are obtained by rigid translation of a Riemannian unit sphere by a vector field $W$. The local and global geometries of these Finslerian metrics have been extensively studied (\cite{SSY}).

Another interesting but much less studied problem is the Matsumoto's slope metric $F=\frac{\alpha^2}{\alpha-\beta}$.

Based on a letter of P. Finsler (1969), M. Matsumoto considered the following problem:\\

{\it Suppose a person walking on a horizontal plane with velocity $c$, while the gravitational force is acting perpendicularly on this plane. The person is almost ignorant of the action of this force. Imagine the person walks now with same velocity on the inclined plane of angle $\varepsilon$ to the horizontal sea level. Under the influence of gravitational forces, what is the trajectory the person should walk in the center to reach a given destination in the shortest time?}\\

Based on this, he has formulated the following
{\it Slope principle} (\cite{M1},\cite{M2}).

{\it With respect to the time measure, a plane $(\pi)$ with an angle $\varepsilon$ inclination can be regarded as a Minkowski plane. The indicatrix curve of the corresponding Minkowski metric is a lima\c{c}on, contained in this plane, given by
		\begin{equation*}
			r=c+a\cos\theta,
		\end{equation*}		
in the polar coordinates $(r,\theta)$ of $(\pi)$, whose pole is the origin $O$ of $(\pi)$ and the polar axis is the most steepest downhill direction, 
 where $a=\frac{g}{2}\sin\varepsilon$, and $g$ is the acceleration constant.

From calculus of variations it follows that for a hiker walking the slope of a mountain under the influence of gravity, the most efficient time minimizing paths are not the Riemannian geodesics, but the geodesics of the slope metric $F=\frac{\alpha^2}{\alpha-\beta}$.
}

More recently, it was shown that the fire fronts evolution can be modeled by Finsler merics of slope type and their generalizations (see \cite{Ma}). In this setting the geodesics befaviour and the cut locus have real interpretations and concrete applications for the firefighters activity as well as preventing of wild fires.  All these applications show that slope metrics deserve a more detalied study making in this way the motivation of the preseant paper. 

Despite the quite long existence of slope metrics, their study is limited mainly to the study of their local geometrical properties, while the global existence of such metrics and other geometrical properties are conspicuously absent. 

Our study leads to the following novel findings:
\begin{enumerate}
\item we show that there are many examples of surfaces admitting globally defined slope metrics;
\item we describe in some detail the geometry of a surface of revolution endowed with a slope metric. In special we study the geodesics behaviour, Clairaut relation, etc.;
\item we compare the Finslerian areas (by using the Busemann-Hausdorff and the Holmes-Thompson volume forms, respectively) with the Riemannian one. 
\end{enumerate} 

\bigskip

Here is the contents of the present paper. We recall in Section \ref{sec_2} the construction of the slope metric on a surface $M\to\R^3$ based on Matsumoto's work pointing out the strongly convexity condition such a surface must satisfy in order to admit a slope metric (Proposition \ref{prop_convex f}).

Based on these we show that there exist smooth surfaces $M\to\R^3$ that admit globally defined slope metrics (Section \ref{sec_3}). All the examples known until now were local one. This is for the first time the existence of global slope metrics is shown.

In Section \ref{sec_4} we specialize to surfaces of revolution admitting globally defined slope metrics. We study in Section \ref{sec_41} general Finsler surfaces of revolution and give a new form of the Clairaut relation in Theorem \ref{thm_F-Clairaut}. This relation is very important showing that the geodesic flow of Finsler surfaces of revolution is integrable despite its highly nonlinear character. After solving the algebraic system \eqref{algebraic} one can write the geodesic equations in an explicit form, however solving this system is not a trivial task. Next, in Section \ref{sec_43}, we construct explicitly the slope metric on a surface of revolution and show that there are many such surfaces admitting globally defined strongly convex slope metrics, see Theorem \ref{thm_new} for a topological classification and examples. These are actually Finsler surfaces of revolution (see Theorem \ref{thm_F_surface of revolution}).

We turn to study of geodesics of slope metrics on a surface of revolution in Section \ref{sec_44} by explicitly writing the geodesic equations as second order ODEs in \eqref{F geodesic equations}. Some immediate consequences are given (see Proposition \ref{prop_meri}, \ref{prop_para}). The meridians are $F$-geodesics, but parallels are not. Moreover, a slope metric cannot be projectively flat or projectively equivalent to Riemannian metric $\alpha$ (Proposition \ref{prop_proj}).
We show the concrete form of the Clairaut relation for this case in Theorem \ref{thm_Clairaut_p2}, and some consequence of it in Proposition \ref{prop_F-Clairaut}. 

Finally, we compare the area of a bounded region $D$ on the surface of revolution $M$ when measured by the canonical Riemannian, Busemann-Hausdorff, and Holmes-Thompson volume measures, respectively (see Theorems \ref{thm: area comparison} and \ref{thm: area comparison2}). 

Other topics in the geometry of slope metrics like the study of the flag curvature, global behaviour of geodesics, and cut locus, etc. will be considered in forthcoming research.
 
\bigskip 

\noindent
{\it Acknowledgments}:  The authors are grateful to Prof. H. Shimada and Prof. M. Tanaka for many useful discussion. We thank to R. Hama as well for his help.

\section{The slope metric}\label{sec_2}

In this section we will construct the slope metric from the movement on a Riemannian surface under the influence of the gravity attraction force.

A hiker is walking on the surface $M$, seen now as the slope of a mountain, with speed $c$ an level ground, along a path that makes an angle $\varepsilon$ with the steepest downhill direction.

Let us consider the surface $M$ embedded in the Euclidean space $\R^3$ with the parametrization
\begin{equation}
M \to \mathbb{R}^3,\quad (x,y)\mapsto (x,y,z=f(x,y)),
\end{equation}
where $f:\R^2\to\R$ is a smooth function (further conditions will be added later), that is $M$ is the graph of $z=f(x,y)$. It is elementary to see that the tangent plane $\pi_p=T_pM$ at a point $p=(x,y,f(x,y))\in M$ is spanned by
$\partial_x:=(1,0,f_x),\quad \partial_y:=(0,1,f_y)$, 
where $f_x$ and $f_y$ are the partial derivatives of $f$ with respect to $x$ and $y$, respectively. The induced Riemannian metric from $\R^3$ to the surface $M$ is

\begin{equation} \label{matrix aij}
a_{ij}=
\left( \begin{array}{cc}
1+f_x^2&f_xf_y\\
f_xf_y&1+f_y^2
\end{array} \right).
\end{equation}

We will construct the slope metric on the surface $M$ by considering
\begin{itemize}
\item the plane $x,y$ to be the sea level;
\item the $z\geq 0$ coordinate to be the altitude above the sea level;
\item the surface $M:z=f(x,y)$ to be the slope of the mountain.
\end{itemize}

At any point $p\in M$ we construct a Riemannian orthonormal frame $\{e_1,e_2\}$ in $T_pM$ by choosing $e_1$ to point on the steepest downhill direction of $T_pM$. Indeed, it is elementary to see that
\begin{equation}
\begin{split}
e_1&=-\frac{1}{\sqrt{(1+f_x^2+f_y^2)(f_x^2+f_y^2)}}(f_x\partial_x+f_y\partial_y)\\
e_2&=\frac{1}{\sqrt{f_x^2+f_y^2}}(-f_y\partial_x+f_x\partial_y).
\end{split}
\end{equation}	
is a such orthonormal frame.

With these notations, the Matsumoto's slope principle is telling us that the locus of unit time destinations 
of the hiker on the plane $T_pM$ is given by the 
lima\c{c}on
\begin{equation}
r=c+a\cdot \cos\theta,
\end{equation}		
where $(r,\theta)$ are the polar coordinates in $T_pM$, $c$ is the speed of the hiker on the ground level $xy$, and $a=\frac{g}{2}\cdot \sin \varepsilon$ is the gravity (of magnitude $g$) component along the steepest downhill direction.
 The Finsler norm $F$ having this lima\c{c}on as indicatrix measures time travel on the surface $S$.
 
Taking into account the parametrization
\begin{equation}
X(t)=(c+a\cos t)\cdot \cos t,\quad 
Y(t)=(c+a\cos t)\cdot \sin t,\quad t\in[0,2\pi)
\end{equation}
it is easy to obtain the implicit equation of the lima\c{c}on
\begin{equation}\label{implicit limacon}
X^2+Y^2=c\sqrt{X^2+Y^2}+a\cdot X,
\end{equation}		
where $X,Y$ are the coordinates with respect to the orthonormal frame $\{e_1,e_2\}$ in $T_pM$.

By Okubo's method (\cite{AIM}), we get the Minkowski norm
\begin{equation*}
F(X,Y)=\frac{X^2+Y^2}{c\sqrt{X^2+Y^2}+a\cdot X},
\end{equation*}
and by converting to the canonical coordinates $(x,y,\dot{x},\dot{y})$ of $TM$ we obtain the slope metric
\begin{equation*}	
\begin{split}
F(x,y,\dot{x},\dot{y})&=\frac{\alpha^2}{c\alpha-\frac{g}{2}\beta},
\end{split}		
\end{equation*}	
where		
\begin{equation}\label{alpha,beta}
\begin{cases}
\alpha&=\sqrt{(1+f_x^2)\dot{x}^2+2f_xf_y\dot{x}\dot{y}+(1+f_y^2)\dot{y}^2}\\
\beta&= f_x\dot{x}+f_y\dot{y}.
\end{cases}
\end{equation}

For the sake of simplicity we can choose $c:=\frac{g}{2}$ and by multiplication with $c$ we obtain the usual form of the slope metric

\begin{equation}\label{slope metric}
F=\frac{\alpha^2}{\alpha-\beta}
\end{equation}
(see \cite{AIM}, \cite{BR}, \cite{M1}).

One can now easily see that the slope metric belongs to the class of $(\alpha,\beta)$-metrics, that is Finsler metrics with fundamental function $F=F(\alpha,\beta)$, where $\alpha^2=a_{ij}y^iy^j$ is a Riemannian metric $(a_{ij})$ on $M$ and $\beta=b_i(x)y^i$ is a linear form in $TM$. For the general theory of $(\alpha,\beta)$-metrics one is referred to \cite{BCS2} or \cite{M1991}.

By writing $F=F(\alpha,\beta)=\alpha\cdot\phi(\mathfrak s)$, where $\mathfrak s=\frac{\beta}{\alpha}$, the Hessian $g_{ij}:=\frac{1}{2}\frac{\partial^2 F^2}{\partial y^i \partial y^j}$ reads

\begin{equation*}
g_{ij}=\rho a_{ij}+\rho_0b_ib_j+\rho_1(b_i\alpha_j+b_j\alpha_i)-\rho\rho_1\alpha_i\alpha_j,
\end{equation*}
where $\alpha_i:=\frac{\partial\alpha}{\partial y^i}$, and
\begin{equation*}
\rho=\phi^2-\mathfrak s\phi\phi',\quad \rho_0=\phi\phi''+\phi'\phi',\quad \rho_1=-\mathfrak s(\phi\phi''+\phi'\phi')+\phi\phi'.
\end{equation*}

It is known from Shen's work (see \cite{BCS2} or \cite{S}) that $(\alpha,\beta)$ type Finsler metrics are strongly convex whenever the function $\phi(\mathfrak s)$ satisfies
\begin{equation*}
\begin{split}
\phi(\mathfrak s)&>0\\
\phi(\mathfrak s)-\mathfrak s\phi'(\mathfrak s)&>0\\
\phi''(\mathfrak s)&\geq 0,\quad \text{for}\quad \mathfrak s<b.
\end{split}
\end{equation*}

In the case of the slope metric, we have $\phi(\mathfrak s)=\frac{1}{1-s}$ and the relations above are clearly satisfied for $\mathfrak s<\frac{1}{2}$, that is $\beta<\frac{1}{2}\alpha$. It follows

\begin{proposition}\label{prop_convex f}
A surface $M\to\R^3$, $(x,y)\mapsto (x,y,z=f(x,y))$ admits a strongly convex slope metric $F=\frac{\alpha^2}{\alpha-\beta}$, where $\alpha,\beta$ are given in \eqref{alpha,beta}, if and only if 

\begin{equation}\label{convex condition}
f^2_x+f^2_y<\frac{1}{3}
\end{equation}
where $f_x$, $f_y$ are partial derivatives of $f$.

\end{proposition}

This proposition is saying that $\beta<\frac{1}{2}\alpha$ is equivalent to the condition \eqref{convex condition}, for  $\alpha,\beta$ given in \eqref{alpha,beta}.

Indeed, if we assume $\beta<\frac{1}{2}\alpha$ is true for any 
 $(x,y,\dot x, \dot y)\in TM$, then by taking $(\dot x, \dot y)$ to be $(b_1,b_2)=(f_x,f_y)$ in this inequality, \eqref{convex condition} follows immediately. 
 Conversely, assume that \eqref{convex condition} is true everywhere on $S$ and prove $\beta<\frac{1}{2}\alpha$ for any $(x,y,\dot x, \dot y)\in TM$. The idea is to use the Cauchy-Schwartz inequality in the Euclidean plane for the vectors $(f_x,f_y)$ and $(\dot x,\dot y)$, that is 
 $$
 f_x\dot x+f_y\dot y\leq \sqrt{(f_x^2+f_y^2)(\dot x^2+\dot y^2)}
 $$
 and by using the hypothesis \eqref{convex condition} it follows the equivalent condition 
 \begin{equation}\label{equiv conv cond}
  f_x\dot x+f_y\dot y < \sqrt{\frac{1}{3}(\dot x^2+\dot y^2)}.
 \end{equation}
From here $\beta<\frac{1}{2}\alpha$ follows immediately. 
 
\begin{remark}
\begin{itemize}
\item[(1)] This formula was obtained for the first time in \cite{BR} and the proof above is based on the idea in \cite{BR}.
\item[(2)] The convexity formula above is obviously equivalent to the usual convexity condition of the lima\c{c}on $c>2a$.
\item[(3)] Taking into account the inverse matrix $(a^{ij})$ of \eqref{matrix aij}, it can be seen that $b^2:=a^{ij}b_ib_j$ is given by 
$$
b^2=\frac{f_x^2+f_y^2}{1+f_x^2+f_y^2},
$$
and from \eqref{convex condition} it follows that the strongly convexity of the indicatrix is equivalent to 
$$
b<\frac{1}{2}.
$$
\end{itemize}
\end{remark}

Observe that for the slope metric \eqref{slope metric} we have

\begin{equation*}
\rho=\frac{2\mathfrak s-1}{(\mathfrak s-1)^3}=\frac{\alpha^2(\alpha-2\beta)}{(\alpha-\beta)^3};\quad \rho_0=\frac{3}{(\mathfrak s-1)^4}=\frac{3\alpha^4}{(\alpha-\beta)^4};\quad \rho_1=\frac{1-4\mathfrak s}{(\mathfrak s-1)^4}=\frac{\alpha^3(\alpha-4\beta)}{(\alpha-\beta)^4}
\end{equation*}
and hence
\begin{equation*}
\begin{split}
g_{11}&=\rho\cdot a_{11}+\rho_0+2\rho_1\alpha_1-\rho
\cdot\rho_1(\alpha_1)^2\\
g_{12}&=(1-\rho\alpha_1)\rho_1\alpha_2\\
g_{22}&=\rho\cdot a_{22}-\rho\cdot\rho_1(\alpha_2)^2.
\end{split}
\end{equation*}


\section{Examples of slope metrics}\label{sec_3}

One might be tempted to think that due to the convexity condition \eqref{convex condition} the slope metric is strongly convex only locally. See of instance the example of the paraboloid of revolution $f(x,y):=100-x^2-y^2$ in \cite{BR} where the strongly convexity condition is assured only in a circular vicinity of the hilltop.

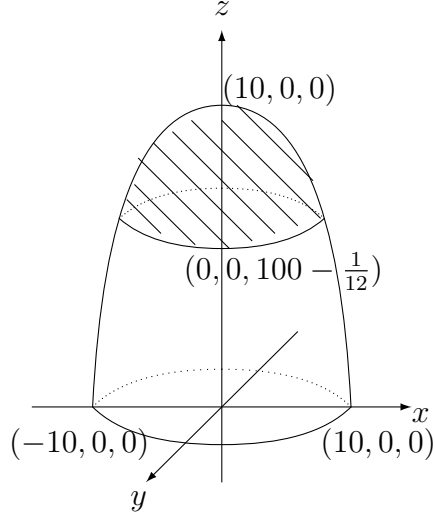
\begin{figure}[H]
\centering
\setlength{\unitlength}{1cm}
\begin{picture}(5, 7)
 
\put(0,1){\vector(1,0){5}}
\put(2.5,0){\vector(0,1){6}}
\put(3.5,2){\vector(-1,-1){2}}

\put(3.7,4){\line(-1,1){1}}
\put(3.8,3.5){\line(-1,1){1.3}}
\put(3.5,3.4){\line(-1,1){1.4}}
\put(3.2,3.3){\line(-1,1){1.35}}
\put(2.9,3.2){\line(-1,1){1.3}}
\put(2.6,3.1){\line(-1,1){1.2}}
\put(2.15,3.15){\line(-1,1){0.8}}

\put(1.7,3.3){\line(-1,1){0.45}}

\qbezier(2.5,5)(4,5)(4.2,1)
\qbezier(2.5,5)(1,5)(0.8,1)

\qbezier(2.5,0.5)(1.3,0.5)(0.8,1)
\qbezier(2.5,0.5)(3.7,0.5)(4.2,1)

\qbezier(2.5,3.1)(3.4,3.1)(3.85,3.5)
\qbezier(2.5,3.1)(1.6,3.1)(1.15,3.5)

\qbezier[25](2.5,1.5)(3.7,1.5)(4.2,1)
\qbezier[25](2.5,1.5)(1.3,1.5)(0.8,1)

\qbezier[25](2.5,3.9)(3.4,3.9)(3.85,3.5)
\qbezier[25](2.5,3.9)(1.6,3.9)(1.15,3.5)

\put(2,2.7){$(0,0,100-\frac{1}{12})$}
\put(2.5,5.1){$(10,0,0)$}
\put(-0.3,0.4){$(-10,0,0)$}
\put(3.8,0.4){$(10,0,0)$}
\put(5,0.8){$x$}
\put(1.3,-0.3){$y$}
\put(2.4,6.2){$z$} 
 
\end{picture}
\bigskip
\caption{The slope metric is strongly convex in a circular vicinity of the hilltop on a paraboloid of revolution.}
\end{figure}

However, that is not the case. There are many Riemannian surfaces that admit globally strongly convex slope metric. We describe few such examples below.

\subsection{The plane}

The simplest surface is the plane $M:z=f(x,y)=px+qy+r$, where $p,q,r$ are constants.

It is trivial to see that
\begin{equation*}
(a_{ij})=\left(\begin{matrix}
1+p^2 & pq \\ pq & 1+q^2
\end{matrix}\right),\quad
(b_i)=\left(
\begin{matrix}
p\\q
\end{matrix}\right),
\end{equation*}
thus the slope metric is actually the Minkowski metric

\begin{equation*}
F=\frac{(1+p^2)\dot{x}^2+2pq\dot{x}\dot{y}+(1+q^2)\dot{y}^2}{\sqrt{(1+p^2)\dot{x}^2+2pq\dot{x}\dot{y}+(1+q^2)\dot{y}^2}-(p\dot{x}+q\dot{y})},
\end{equation*}
with the strongly convexity condition
\begin{equation*}
p^2+q^2<\frac{1}{3}.
\end{equation*}

Hence the plane $z=\lambda\cdot x$ does admit a strongly convex slope metric for any constant $\lambda^2<\frac{1}{3}$, while $z=x$ does not (see \cite{BR}).

We recall from \cite{SS} that for a slope metric on a surface $M:z=f(x,y)$, the 1-form $\beta$ is parallel with respect to $\alpha$ if and only if $M$ is a plane. In this case the slope metric is a Berwald space.

\subsection{A list of surfaces}

Elementary computations show that all the following surfaces $z=f(x,y)$ admit strongly convex slope metric globally defined, where $f:\R^2\to\R$ are given by

\begin{enumerate}
\item $f(x,y)=\frac{1}{2\sqrt{6}}e^{-(x^2+y^2)}$,
\item $f(x,y)=\frac{1}{2\sqrt{6}}e^{-(x+y)^2}$,
\item $f(x,y)=\frac{1}{2\sqrt{6}}\arctan(x+y)$,
\item $f(x,y)=\frac{1}{2\sqrt{6}}((x+y)-\log(e^{x+y}+1))$,
\item $f(x,y)=\frac{1}{2\sqrt{6}}\log(\sqrt{(x+y)^2+1}+x+y)$.
\end{enumerate}

Indeed, all these surfaces satisfy condition \eqref{convex condition} for any $(x,y)\in\R^2$.

\begin{figure}[H]
    \centering
    \begin{subfigure}[b]{0.3\textwidth}
        \includegraphics[width=\textwidth]{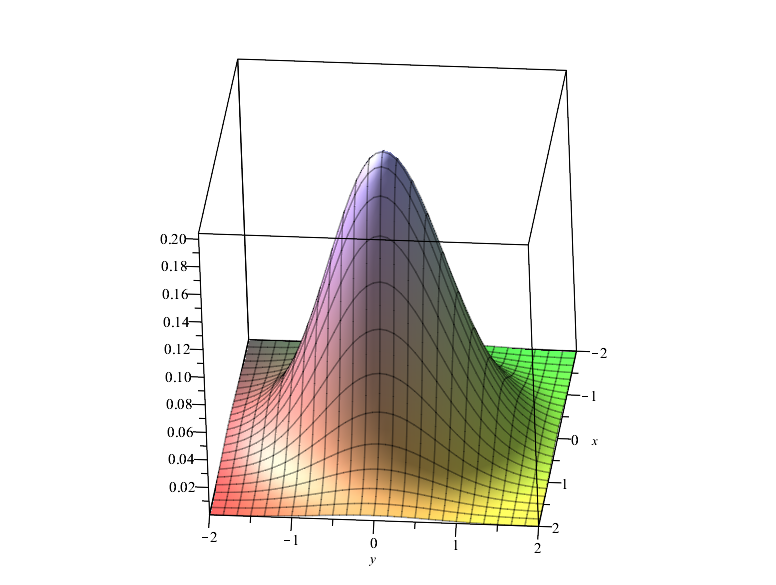}
        \caption{$\frac{1}{2\sqrt{6}}e^{-(x^2+y^2)}$}
    \end{subfigure}\quad
    \begin{subfigure}[b]{0.2\textwidth}
        \includegraphics[width=\textwidth]{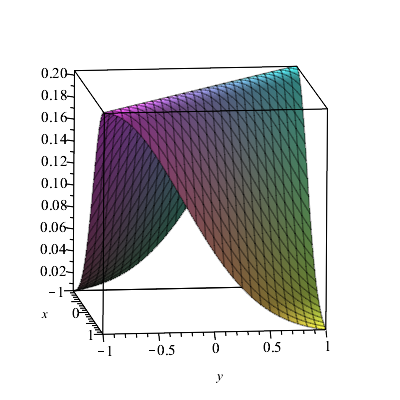}
        \caption{$\frac{1}{2\sqrt{6}}e^{-(x+2)^2}$}
        \label{fig:tiger}
    \end{subfigure}\quad
    \begin{subfigure}[b]{0.25\textwidth}
        \includegraphics[width=\textwidth]{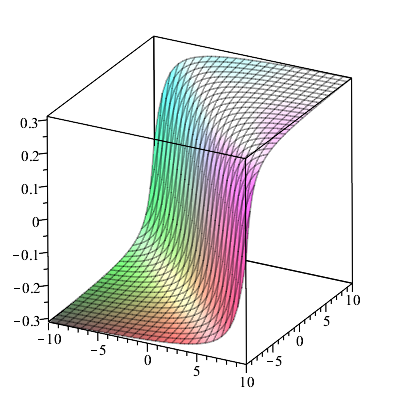}
        \caption{$\frac{1}{2\sqrt{6}}\arctan(x+y)$}
        \label{fig:mouse}
    \end{subfigure}
    \caption{Graphs of surfaces 1,2,3 described above.}\label{fig:animals}
\end{figure}

Let us remark that the surface $f(x,y):=\frac{1}{2\sqrt{6}}e^{-(x^2+y^2)}$ can be actually realized as a surface of revolution obtained by rotating the graph of the function $z=\frac{1}{2\sqrt{6}}e^{-x^2}$ around the $z$ axis.

This example suggests that surfaces of revolution are good candidates for the study of slope metrics, fact motivating the next section.

\section{The slope metric of a surface of revolution}\label{sec_4}

\subsection{Riemannian surface of revolution}\label{sec_41}

In order to fix the ideas, let us recall some basic facts from the geometry of Riemannian surfaces of revolution (see \cite{SST}).

A surface of revolution $M\to\R^3$ can be parametrization as

\begin{equation}\label{surface of revolution}
(u,v)\mapsto (x=m(u)\cos v,y=m(u)\sin v,z=u)
\end{equation}
where $u\in(0, \infty)$, $v\in\Sph^1$. Here $(u,v)$ are the geodesic polar coordinates around the pole $p\in M$, and  $m:(0,\infty)\to (0,\infty)$ is a smooth function such that $m'(0)=1$ (see \cite{SST} for details).  

\begin{remark}
	We have defined here a classical surface of revolution by rotating the image of the curve   $m:(0,\infty)\to (0,\infty)$ around the $z$ axis. However, there is no harm in taking $m:I\to (0,\infty)$, where $I\subset \R$ is an open set. See the examples below. 
\end{remark}

It is known that a curve
\begin{itemize}
\item $u=u(t)$, $v=v_0$: constant is called a {\it meridian}, and
\item $u=u_0$: constant, $v=v(t)$ is called a {\it parallel}.
\end{itemize}

Recall that a point $p\in M$ is called {\it pole} if any 2 geodesics emanating from $p$ do not meet again, in other words, the cut locus of $p$ is empty. A unit speed geodesic is called a {\it ray} if $d(\gamma(0),\gamma(s))=s$, for all $s\geq 0$. Clearly, all geodesics emanating from the pole are rays.

The induced Riemannian metric is
\begin{equation}\label{Riemannian metric}
(a_{ij})=\left(\begin{matrix}
1+(m')^2(u) & 0 \\ 0 & m^2(u)
\end{matrix}\right)
\end{equation}
and the unit speed geodesics $(u=u(t),v=v(t))$ are given by

\begin{equation}
\begin{cases}
\frac{d^2u}{dt^2}+\frac{m'm''}{1+m'^2}\left(\frac{du}{dt}\right)^2-\frac{mm'}{1+m'^2}\left(\frac{dv}{dt}\right)^2=0\\
\frac{d^2v}{dt^2}+2\frac{m'}{m}\frac{du}{dt}\frac{dv}{dt}=0.
\end{cases}
\end{equation}

The geodesic spray coefficients of this Riemannian metric read

\begin{equation*}
\begin{cases}
2\mathcal{G}^1_\alpha=\frac{m'm''}{1+m'^2}(y^1)^2-\frac{mm'}{1+m'^2}(y^2)^2\\
2\mathcal{G}^2_\alpha=2\frac{m'}{m}y^1y^2,\quad (m\neq 0).
\end{cases}
\end{equation*}

From here it follows that there exists a constant $\nu$, called the {\it Clairaut constant} such that

\begin{equation}\label{Clairaut relation}
\frac{dv}{dt}\cdot m^2(u(t))=\nu,
\end{equation}
{and hence$\frac{du}{dt}=\pm\frac{1}{m}\sqrt{\frac{m^2-\nu^2}{1+m'^2}}$, that is in the case of a Riemannian surface of revolution, the geodesic flow is integrable.

\begin{remark}
It is known that by changing the parameter $u$ on the profile curve $m(u)$ it is possible to parametrize $M$ as $(u,v)\mapsto (m(u)\cos v,m(u)\sin v,z(u))$ such that $[m'(u)]^2+[z'(u)]^2=1$. This leads to simple form of the induced Riemannian metric $(a_{ij})=\left(\begin{matrix}
1 & 0 \\ 0 & m^2(u)
\end{matrix}\right)$. We are not using this parametrization because linear form $\beta$ in \eqref{alpha,beta} is simpler when using \eqref{surface of revolution} and this leads to simplication of computations for the slope metric.
\end{remark}

\subsection{Finsler surfaces of revolution}\label{sec_42}

Let $(M,F)$ be a Finsler structure defined on a surface of revolution $M$ defined as in Section 4.1.

If $X:=\frac{\partial}{\partial v}$ is a Killing vector field for $F$, that is $\mathcal{L}_XF=X^c(F)=0$, where $X^c$ is the complete lift of $X$ to $TM$, or equivalently $\frac{\partial F}{\partial v}=0$, then $(M,F)$ is called a {\it Finsler surface of revolution}.

\begin{remark}
\begin{enumerate}
\item See \cite{INS} for a definition based on the notion of motion. Their definition is equivalent to ours.
\item See \cite{HCS} and \cite{HKS} for a complete study of rotationally Randers metrics, that is Finsler metric of type $F=\alpha+\beta$ constructed on surfaces of revolution. 
\end{enumerate}
\end{remark}

If we denote by $H(x,p)$ the Hamiltonian corresponding to the Finsler structure $(M,F)$ by means of Legendre transform (see \cite{MHSS}), then since $F$ is surface of revolution, it follows $\frac{\partial H}{\partial v}=0$. Hence, Hamilton Jacobi equations $\frac{dx^i}{ds}=\frac{\partial H}{\partial p_i}$, $\frac{dp_i}{ds}=-\frac{\partial H}{\partial x^i}$ imply that $I=p_2$ is a prime integral of the geodesic flow, that is $\frac{dp_2}{ds}=0$ along any unit speed $F$-geodesic.

On the other hand, recall that by the Legendre transform associated to $F$, we have
\begin{equation}
p_2=g_{2i}y^i=g_{12}y^1+g_{22}y^2
\end{equation}
and hence we obtain

\begin{theorem}\label{thm_F-Clairaut}
Along any unit speed $F$-geodesics $\mathcal{P}(s)=(u(s),v(s))$ we have
\begin{equation}\label{F-Clairaut}
p_2(s)=g_{12}(\mathcal{P},\dot{\mathcal{P}})\cdot\frac{du}{ds}+g_{22}(\mathcal{P},\dot{\mathcal{P}})\cdot\frac{dv}{ds}=\nu_F=\text{constant}.
\end{equation}
\end{theorem}

That is, \eqref{F-Clairaut} is the corresponding relation to \eqref{Clairaut relation} in the Finslerian setting.

The constant $\nu_F$ plays the role of the Clairaut constant for Finslerian geodesics.

\begin{remark}
See \cite{INS} for an alternate proof of this formula.
\end{remark}

It follows that, for any unit speed $F$-geodesic, we have
\begin{equation}\label{algebraic}
\begin{cases}
g_{12}(\mathcal{P},\dot{\mathcal{P}})\cdot\frac{du}{ds}+g_{22}(\mathcal{P},\dot{\mathcal{P}})\cdot\frac{dv}{ds}=\nu_F\\
F(\mathcal{P},\dot{\mathcal{P}})=1
\end{cases}
\end{equation}
and theoretically, by solving this algebraic system, we can obtain $\frac{du}{dt}$ and $\frac{dv}{dt}$ that by integration would give the trajectories of the $F$-geodesics. However, observe that finding an explicit solution of the system is not a trivial task.

\begin{remark}
\begin{enumerate}
\item As far as we know, the relation \eqref{F-Clairaut} appeared for the first time in the case of the rotational Randers surface of revolution studied in \cite{HCS}, where the Clairaut constant for the Randers geodesics is $\frac{\nu}{1+\mu\nu}$. Here $\nu$ is the usual Clairaut constant of the corresponding Riemannian geodesic through the Zermelo navigation process.
\item  We denote by $\varphi_t$ the flow of $\frac{\partial}{\partial v}$, which is a Finslerian isometry preserving the orientation of $M$. The Finslerian distance $d_F$ is invariant under $\varphi_t$.
\end{enumerate}
\end{remark}

\subsection{The slope metric on a surface of revolution}\label{sec_43}

Let us consider again the surface of revolution $M$ with the parametrization \eqref{surface of revolution} and induced Riemannian metric \eqref{Riemannian metric}.

Following again Matsumoto's slope principle, observe that the orthonormal frame in $T_pM$ at a given $p\in M$ is

\begin{equation*}
\begin{cases}
e_1=-\frac{1}{\sqrt{(m')^2+1}}\cdot\frac{\partial}{\partial u}\\
e_2=\frac{1}{m}\cdot\frac{\partial}{\partial v}
\end{cases}
\end{equation*}
and here, the relation between the coordinates $(X,Y)$ of $T_pM$ with respect to $\{e_1,e_2\}$ and the canonical coordinates $(\dot{u},\dot{v})$ is
\begin{equation*}
X=-\sqrt{1+(m')^2}\cdot\dot{u},\quad Y=m\cdot\dot{v}.
\end{equation*}

The lima\c{c}on implicit equation \eqref{implicit limacon} reads now
\begin{equation*}
\left[1+(m')^2\right]\dot{u}^2+m^2\cdot\dot{v}^2=c\sqrt{\left[1+(m')^2\right]\dot{u}^2+m^2\dot{v}^2}-a\sqrt{1+(m')^2}\cdot\dot{u},
\end{equation*}
and taking into account that $a=\sin\varepsilon=\frac{1}{\sqrt{1+(m')^2}}$ we obtain the slope metric in the form \eqref{slope metric} with
\begin{equation}\label{slope metric on surface of revolution}
\begin{cases}
\alpha=\sqrt{\left[1+(m')^2\right]\dot{u}^2+m^2\dot{v}^2}\\
\beta=\dot{u}.
\end{cases}
\end{equation}

Taking into account the strongly convexity condition $b<\frac{1}{2}$ it follows

\begin{theorem}\label{thm_F_surface of revolution}
A surface of revolution $M\to\R^3$, $(u,v)\mapsto(m(u)\cos v,m(u)\sin v,u)$ admits a strongly convex slope metric $F=\frac{\alpha^2}{\alpha-\beta}$, with $\alpha,\beta$ given in \eqref{slope metric on surface of revolution} if and only if
\begin{equation}\label{convex condition on surface of revolution}
(m')^2>3.
\end{equation}

Moreover, $(M,F)$ is a Finsler surface of revolution.
\end{theorem} 

Let us recall from Poincar\'e -Hopf index theorem for the rotational vector filed $X=\frac{\partial}{\partial v}\Big\vert_{p}$, $p\in M$, that the strongly convexity condition \eqref{convex condition on surface of revolution} implies that number of singular points of $X$ on $M$ can be only 1 or 0. Indeed, otherwise $X$ would be vanishing, or $M$ would be homeomorphic to the sphere, and this is not possible.
It is clear from \eqref{convex condition on surface of revolution} that $M$ cannot be boundaryless compact manifold. The case of a cylinder of revolution is not possible either due   \eqref{convex condition on surface of revolution}, hence we obtain

\begin{theorem}\label{thm_new}
The surfaces of revolution $M$ admitting globally defined strongly convex slope metrics one  homeomorphic to $\R^2$.
\end{theorem}

One can now easily construct examples of surfaces of revolution satisfy condition \eqref{convex condition on surface of revolution}. Here are such surfaces

\begin{enumerate}
\item $m(u)=\sqrt{6u^2-1}$, for $u\in(\frac{1}{\sqrt{6}},\infty)$;
\item $m(u)=\frac{1}{2}\sqrt{-2\ln(24u^2)}$, for $u\in(0,\frac{1}{2\sqrt{6}})$.
\end{enumerate}

\begin{figure}[H]
    \centering
   
    \begin{subfigure}[b]{0.3\textwidth}
        \includegraphics[width=\textwidth]{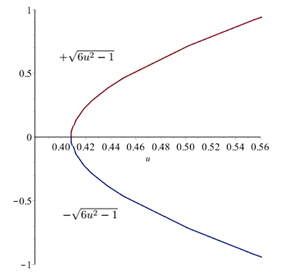}
        \caption{$\pm\sqrt{6u^2-1}$}
    \end{subfigure}\quad
          \begin{subfigure}[b]{0.3\textwidth}
        \includegraphics[width=\textwidth]{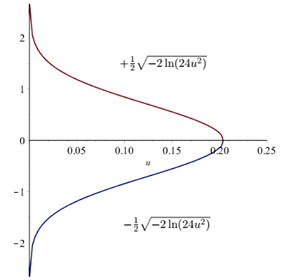}
        \caption{$\pm\frac{1}{2}\sqrt{-2\ln(24u^2)}$}
        \end{subfigure}\quad\quad
    \caption{Graphs of profile curves corresponding to the functions $m(u)$ in the examples above. Pay attention to the fact that this are actually the graph of the inverse function $m^{-1}$.}
\end{figure}
Since the slope metric $F$ is a Finslerian surface of revolution, the theory explained in Section 4.2 applies.

\subsection{The geodesics of a surface of revolution with the slope metric}\label{sec_44}

In order to study to geodesics of the slope metric $(M,F=\frac{\alpha^2}{\alpha-\beta})$ we need a formula for the geodesic spray of $F$.

We recall the general formula for an arbitrary $(\alpha,\beta)$-metric

\begin{equation*}
G^i=\mathcal{G}^i_\alpha+\alpha Qs^i_0+\Theta\{-2Q\alpha s_0+r_{00}\}\frac{y^i}{\alpha}+\Psi\{-2Q\alpha s_0+r_{00}\}b^i
\end{equation*}
where $G^i$ and $\mathcal{G}^i_\alpha$ denote the spray coefficients for $F$ and $\alpha$, respectively.

Here we use the customary notations:

\begin{equation*}
\begin{split}
r_{ij}:=\frac{1}{2}(b_{i|j}+b_{j|i}),\quad s_{ij}:=\frac{1}{2}(b_{i|j}-b_{j:i})\\
s^i_j:=a^{ik}s_{kj},\quad s_j=b_is^i_j,\quad b^i=a^{ij}b_j,
\end{split}
\end{equation*}
and
\begin{equation*}
\begin{split}
Q&:=\frac{\phi'}{\phi-\mathfrak s\phi'}\\
\Theta&:=\frac{\phi-\mathfrak s\phi'}{2[\phi-\mathfrak s\phi'+(b^2-\mathfrak s^2)\phi]}\cdot\frac{\phi'}{\phi}-\mathfrak s\Psi=\left[\frac{\phi-\mathfrak s\phi'}{\phi''}\cdot\frac{\phi'}{\phi}-\mathfrak s\right]\Psi\\
\Psi&:=\frac{\phi''}{2[\phi-\mathfrak s\phi'+(b^2-\mathfrak s^2)\phi'']}
\end{split}
\end{equation*}
(see \cite{BCS2}).

In the case of $\alpha,\beta$ given in \eqref{slope metric on surface of revolution} we obtain

\begin{equation*}
\left(\begin{matrix}
b_{1|1} \\ b_{2|2}
\end{matrix}\right)=
\left(\begin{matrix}
\frac{-m'm''}{1+m'^2}\\
\frac{mm'}{1+m'^2}
\end{matrix}\right),\quad
\begin{matrix}
r_{00}=-2\cdot\mathcal{G}^1_\alpha &\\
s^i_j=s_j=0, &\quad b_{1|2}=0,
\end{matrix}
\end{equation*}
and hence
\begin{equation*}
G^i=\mathcal{G}^i_\alpha+r_{00}\left[\Theta\frac{y^i}{\alpha}+\Psi\cdot b^i\right].
\end{equation*}

By taking into account now $\phi(\mathfrak s)=\frac{1}{1-\mathfrak s}$ after some computations we get

\begin{equation*}
\begin{split}
\Psi&=\frac{1}{2b^2-3\mathfrak s+1}=\frac{\alpha}{(2b^2+1)\alpha-3\beta}\\
\Theta&=\left(\frac{1}{2}-2\mathfrak s\right)\Psi=\frac{1-4\mathfrak s}{2(2b^2-3\mathfrak s+1)}=\frac{\alpha-4\beta}{2\alpha}\cdot\Psi=\frac{\alpha-4\beta}{2[(2b^2+1)\alpha-3\beta]}
\end{split}
\end{equation*}
and hence
\begin{equation*}
G^i=\mathcal{G}^i_\alpha+r_{00}\left[\frac{\alpha-4\beta}{2\alpha}\cdot\frac{y^i}{\alpha}+b^i\right]\cdot\Psi.
\end{equation*}

In particular

\begin{equation*}
\begin{cases}
G^1=\mathcal{G}^1_\alpha\cdot\frac{(\alpha-2\beta)^2}{\alpha[(2b^2+1)\alpha-3\beta]}\\
G^2=\mathcal{G}^2_\alpha-\mathcal{G}^1_\alpha\cdot\frac{\alpha-4\beta}{\alpha[(2b^2+1)\alpha-3\beta]}\cdot y^2,
\end{cases}
\end{equation*}
and therefore, the unit speed $F$-geodesic equations are
\begin{equation}\label{F geodesic equations}
\begin{cases}
\frac{d^2u}{ds^2}+2\mathcal{G}^1_\alpha\cdot\frac{(\alpha-2\beta)^2}{\alpha[(2b^2+1)\alpha-3\beta]}\Bigg\vert_{(u(s),v(s))}=0\\
\frac{d^2v}{ds^2}+2\mathcal{G}^2_\alpha-2\mathcal{G}^1_\alpha\cdot\frac{\alpha-4\beta}{\alpha[(2b^2+1)\alpha-3\beta]}\cdot \frac{dv}{ds}\Bigg\vert_{(u(s),v(s))}=0.
\end{cases}
\end{equation}

The geodesic equations in this form are not of much use. 

However, some conclusions can be drawn.

\begin{proposition}\label{prop_meri}
The meridians are $F$-unit speed geodesics.
\end{proposition}
\begin{proof}
If we consider an ($F$-unit speed) meridian $\mathcal{P}(s)=(u(s),v_0)$, then $\dot{\mathcal{P}}(s)=\left(\frac{du}{ds},0\right)$ and by using the $F$-unit speed condition the geodesic equation \eqref{F geodesic equations} are identically satisfied. 

$\qedd$
\end{proof}

\begin{proposition}\label{prop_para}
A parallel $\mathcal{P}(s)=(u_0,v(s))$ is $F$-geodesic if and only if $m'(u_0)=0$, that is a strongly convex slope metric do not admit parallels geodesics.
\end{proposition}

\begin{proof}
$(\Rightarrow)$ If the parallel $\mathcal{P}(s)=(u_0,v(s))$ is a unit speed $F$-geodesic, then along $\mathcal{P}(s)$, $\alpha^2\Big\vert_{(\mathcal{P},\dot{\mathcal{P}})}=1$ and $\beta\Big\vert_{(\mathcal{P},\dot{\mathcal{P}})}=0$, hence the conclusion follows from the same arguments as in the Riemannian case.

$(\Leftarrow)$ If we assume $m'(u_0)=0$ then the conclusion follows in a similar way with the Riemannian case.

$\qedd$
\end{proof}

\begin{proposition}\label{prop_proj}
The slope metric $F=\frac{\alpha^2}{\alpha-\beta}$ can not be projectively equivalent to the Riemannian metric of $M$, nor projectively flat.
\end{proposition}

\begin{proof}
Recall that a Matsumoto metric $F=\frac{\alpha^2}{\alpha-\beta}$ is projectively equivalent to the Riemannian metric $\alpha$ if and only if $\beta$ is parallel with respect to $a_{ij}$, that is $b_{i|j}=0$, where $|$ is the covariant derivative with respect to $a_{ij}$.

However, observe that in the case of the slope metric we have

\begin{equation*}
b_{1|1}=-\gamma^1_{11}\neq 0,\quad b_{1|2}=0,\quad b_{2|2}=-\gamma^1_{22}\neq 0.
\end{equation*}

In order to be projectively flat $\beta$ must be parallel and $\alpha$ projectively flat. Clearly, none of these conditions is true  in the case of the slope metric.

$\qedd$
\end{proof}

Let us consider the prime integral $p_2$ of the geodesic flow.

A straightforward computation shows

\begin{theorem}\label{thm_Clairaut_p2}
Along the unit speed $F$-geodesic $\mathcal{P}:(0,a)\to M$, $\mathcal{P}(s)=(u(s),v(s))$, $\frac{du}{ds}\neq 0$ for $s\in(0,a)$ we have:
\begin{equation}\label{p2 slope metric}
p_2(s)=(g_{12}y^1+g_{22}y^2)\big\vert_{(\mathcal{P},\dot{\mathcal{P}})}=\rho \big\vert_{(\mathcal{P},\dot{\mathcal{P}})}\cdot m^2(u(s))\cdot\frac{dv}{ds}=\nu_F.
\end{equation}
\end{theorem}

Therefore {$\frac{du}{ds}$, $\frac{dv}{ds}$} are solutions of the following algebraic system:
\begin{equation}\label{algebraic system}
\begin{cases}
\rho\big\vert_{(\mathcal{P},\dot{\mathcal{P}})}\cdot m^2(u(s))\cdot\frac{dv}{ds}=\nu_F\\
\frac{\alpha^2}{\alpha-\beta}\big\vert_{(\mathcal{P},\dot{\mathcal{P}})}=1,
\end{cases}
\end{equation}
{where {$\rho(\mathcal P(s),\dot{\mathcal P}(s) )=\frac{\alpha-2\beta}{(\alpha-\beta)^2}|_{(\mathcal P(s),\dot{\mathcal P}(s) )}$}}. An explicit solution of this algebraic system involves solving a 4$^{th}$ order equation, of type $AX^4+BX^3+CX^2+D=0$, leading to a formula too complicated to be written in here, but this computation is always possible.

Instead of writing the explicit solution of \eqref{algebraic system} we point out some consequence of \eqref{p2 slope metric}.

If a unit speed $F$-geodesic $\mathcal{P}(0,a)\to M$ is tangent to the Killing vector field at its end points, that is
\begin{equation*}
\begin{split}
\dot{\mathcal{P}}(0)&=\frac{1}{F\left(\mathcal{P}(0),\frac{\partial}{\partial v}\Big\vert_{\mathcal{P}(0)}\right)}\cdot \frac{\partial}{\partial v}\Big\vert_{\mathcal{P}(0)}\quad \text{and}\\
\dot{\mathcal{P}}(a)&=\frac{1}{F\left(\mathcal{P}(a),\frac{\partial}{\partial v}\Big\vert_{\mathcal{P}(a)}\right)}\cdot \frac{\partial}{\partial v}\Big\vert_{\mathcal{P}(a)},
\end{split}
\end{equation*}
and $\dot{\mathcal{P}}(s)$ is linearly dependent with $\frac{\partial}{\partial v}\Big\vert_{\mathcal{P}(s)}$ for any $s\in(0,a)$, then Clairaut relation \eqref{p2 slope metric} implies

\begin{proposition}\label{prop_F-Clairaut}
If $\mathcal{P}:(0,a)\to M$, $\mathcal{P}(s)=(u(s),v(s))$ is an $F$-unit speed geodesic such that $\dot{\mathcal{P}}(0)$ and $\dot{\mathcal{P}}(a)$ are linear dependent vectors with $\frac{\partial}{\partial v}\Big\vert_{\mathcal{P}(0)}$ and $\frac{\partial}{\partial v}\Big\vert_{\mathcal{P}(a)}$, respectively, then
\begin{equation*}
m(u(0))=m(u(a))=\nu_F.
\end{equation*}

Moreover, $m(u(s))>m(u(0))$ for $s\in(0,a)$.
\end{proposition}

\section{Finslerian volumes}\label{sec_5}
It is known that the {\it Euclidean volume form} in $\R^n$ is the $n$-form
\[
dV_{\R^n}:=dx^1dx^2\dots dx^n,
\]
and the {\it Euclidean volume} of a bounded open set $D\subset \R^n$ is given by 
\[
Vol(D)=\int_D dV_{\R^n}=\int_D dx^1dx^2\dots dx^n.
\]

Obviously, if $D\subset \R^n$ is a bounded open set, $Vol(D)$ is a finite constant. 

More generally, let us consider a Riemannian manifold $(M,g)$ with the {\it Riemannian volume form} 
\[
dV_g:=\sqrt{g}dx^1dx^2\dots dx^n,
\]
and hence the {\it Riemannian volume} of $(M,g)$ can be computed as 
\[
Vol(M,g)=\int_M dV_g=\int_M \sqrt{g}dx^1dx^2\dots dx^n=\int_M \theta^1\theta^2\dots\theta^n,
\]
where $\{\theta^1,\theta^2,\dots,\theta^n\}$ is a $g$-orthonormal co-frame on $M$, and $g=\det(g_{ij})$.


In general, a {\it volume form} $d\mu$ on an $n$-dimensional  Finsler manifold $(M,F)$ is a globally defined, non-degenerate $n$-form on $M$.  In local coordinates we can always write
\begin{equation}
d\mu=\sigma(x) dx^1\wedge\dots\wedge dx^n,
\end{equation}
where $\sigma$ is a positive function on $M$. 

The usual Finslerian volumes are obtain by different choices of the function $\sigma(x)$. Here are two of the most well studied Finslerian volumes.

\bigskip

The {\it Busemann-Hausdorff} volume form is defined as 
\begin{equation}
dV_{BH}:=\sigma_{BH}(x) dx^1\wedge\dots\wedge dx^n,
\end{equation}
where 
\begin{equation}
\sigma_{BH}(x):=\frac{Vol(\B^n(1))}{Vol(\mathcal B^n_xM)},
\end{equation}
here $\B^n(1)$ is the Euclidean unit $n$-ball, $\mathcal B_x^nM=\{y:F(x,y)=1\}$ is the Finslerian ball and $Vol$ the canonical Euclidean volume. 

This volume form allows us to define the {\it Busemann-Hausdorff} volume of the Finsler manifold $(M,F)$ by 
$$
vol_{BH}(M,F)=\int_M dV_{BH}.
$$

\begin{remark}
	 Observe that the $n$-ball Euclidean volume is
		$$
		Vol(\mathbb B^n(1))=\frac{1}{n}Vol(\Sph^{n-1})=\frac{1}{n}Vol(\Sph^{n-2})\int_0^\pi \sin^{n-2}(t)dt.
		$$
	
\end{remark}
\bigskip

Another volume form naturally associated to a Finsler structure is the {\it Holmes-Thompson} volume form  defined by
\begin{equation}
dV_{HT}=\sigma_{HT}(x)dx^1,...dx^n,
\end{equation}
where 
\begin{equation}
\sigma_{HT}(x):=\frac{Vol(\mathcal B_x^nM,{g_x})}
{Vol(\mathbb B^n(1))}=
\frac{1}{Vol(\mathbb B^n(1))}\int_{\mathcal B_x^nM}(\det g_{ij}(x,y))dy^1...dy^n,
\end{equation}
and the  {\it Holmes-Thompson} volume of the Finsler manifold $(M,F)$ is defined as 
$$
vol_{HT}(M,F)=\int_M dV_{HT}.
$$

\begin{remark}
	\begin{enumerate}
		\item 	If $(M,F)$ is an absolute homogeneous Finsler manifold, then
		 the Busemann-Hausdorff volume is a Hausdorff measure of $M$, and we have
		$$
		vol_{BH}(M,F)\geq vol_{HT}(M,F).
		$$
		(see \cite{Du}).
		
		\item If $(M,F)$ is not absolute homogeneous, then the inequality above is not true anymore. Indeed, for instance let $(M,F=\alpha+\beta)$ be a Randers space. Then, one can easily see that 
		$$
		vol_{BH}(M,F)=\int_M (1-b^2(x))dV_\alpha\leq vol(M,\alpha)=vol_{HT}(M,F),
		$$
		where $b^2(x)=a_{ij}(x)b^ib^j$, and $vol(M,\alpha)$ is the Riemannian volume of $M$ (see \cite{S}). 
\end{enumerate}
\end{remark}
		
		In the case of an Finsler $(\alpha,\beta)$-metric, one can compute explicitly the Finslerian volume in terms of the Riemannian volume (see \cite{BCS}). Indeed, if $(M,F(\alpha,\beta))$ is an $(\alpha,\beta)$-metric on an $n$-dimensional manifold $M$, one denotes 
	\begin{equation}\label{vol coeff functions}
	\begin{split}
	&f(b):=\frac{\int_0^\pi \sin^{n-2}(t)dt}{\int_{0}^\pi\frac{\sin^{n-2}(t)}{\phi(b\cos(t))^n}dt}\\
	& g(b):=\frac{\int_0^\pi\sin^{n-2}(t) T(b\cos t)dt }{\int_0^\pi \sin^{n-2}(t)dt},
	\end{split}
	\end{equation}
where $F=\alpha \phi(\mathfrak s)$, $\mathfrak s=\beta \slash \alpha$, and
$$
T(\mathfrak s):=\phi(\phi-\mathfrak s\phi')^{n-2}[(\phi-\mathfrak s\phi')+(b^2-\mathfrak s^2)\phi''].
$$ 

Then the Busemann-Hausdorff and Holmes-Thompson volume forms are given by 
$$
dV_{BH}=f(b)dV_\alpha,\textrm{ and } dV_{HT}=g(b)dV_\alpha,
$$		
respectively, where $dV_\alpha$ is the Riemannian volume form. 		
		
It is remarkable that if the function $T(\mathfrak s)-1$ is an odd function of $\mathfrak s$, then $dV_{HT}=dV_\alpha$. This is the case of Randers metrics (see \cite{BCS}), but not the case of the slope metric.  

The following lemma is elementary.
\begin{lemma}\label{lem: monoton f,g,h}
	Let us consider the following functions
	\begin{enumerate}
		\item $f:(0,\frac{1}{2})\to (\frac{8}{9},1)$, $f(b):=\frac{2}{2+b^2}$,
		\item $g:(0,\frac{1}{2})\to (\frac{5\sqrt{3}}{9},1)$, $g(b):=\frac{(2-3b^2)}{2(1-b^2)\sqrt{1-b^2}}$,
		\item $h:(0,\frac{1}{2})\to (1,{\frac{5\sqrt{3}}{8}})$, $h(b):=\frac{(2+b^2)(2-3b^2)}{4(1-b^2)\sqrt{1-b^2}}$.
	\end{enumerate}
Then, $f$ and $g$ are both monotone decreasing while $h$ is monotone increasing on the given intervals. 	
\end{lemma}

A direct application of this lemma is the following theorem.

	\begin{theorem}\label{thm: area comparison}
		Let $(M,F)$ be a slope metric on a surface of revolution. Then
			\begin{equation*}
			Area_{BH}(D)<Area_{HT}(D)<Area_\alpha(D)
			\end{equation*}
		for any bounded region $D\subset M$.
	\end{theorem}
	\begin{proof}
			Firstly, observe that in the case of a slope metric, formulas \eqref{vol coeff functions} imply
		\begin{equation*}
		\begin{split}
				f(b) & =\frac{\pi}{\int_0^\pi(1-b\cos t)^2 \cdot dt}=\frac{2}{2+b^2} \\
	g(b) & =\frac{(2-3b^2)}{2(1-b^2)\sqrt{1-b^2}}.
		\end{split}
		\end{equation*}
		Indeed, we have
		\begin{equation*}
		g(b)=\frac{1}{\pi}\int_0^\pi T(b\cos t)dt=\frac{1}{\pi}\int_{-b}^b\frac{T(\tau)}{\sqrt{b^2-\tau^2}}\cdot d\tau,
		\end{equation*}
		where we use the substitution $\tau=b\cdot \cos t$.

		If we write
		\begin{equation*}
		\frac{T(\tau)}{\sqrt{b^2-\tau^2}}=\frac{-1+2\tau}{(\tau-1)^3}\cdot\frac{1}{\sqrt{b^2-\tau^2}}+\frac{2\sqrt{b^2-\tau^2}}{(\tau-1)^4}
		\end{equation*}
		it is not difficult to see that
		\begin{equation*}
		\begin{split}
		&\int_{-b}^b \frac{-1+2\tau}{(\tau-1)^3}\cdot\frac{1}{\sqrt{b^2-\tau^2}}\cdot d\tau=\frac{(2-5b^2)\pi}{2(b^2-1)\sqrt{1-b^2}},\quad \text{and}\\
		&\int_{-b}^b \frac{2\sqrt{b^2-\tau^2}}{(\tau-1)^4}\cdot d\tau = \frac{\pi b^2}{\sqrt{1-b^2}(1-b)^2},
		\end{split}
		\end{equation*}
		hence formula for $g(b)$ follows. Therefore, the functions $f$ and $g$ given in Lemma 
		\ref{lem: monoton f,g,h} are exactly those defined by \eqref{vol coeff functions} in the case of the slope metric.

		It results
		\begin{equation*}
		\begin{split}
		dV_{BH} & =f(b)dV_\alpha=\frac{2}{2+b^2}\cdot dV_\alpha,\\
		dV_{HT} & =g(b)dV_\alpha=\frac{(2-3b^2)}{2(1-b^2)\sqrt{1-b^2}}\cdot dV_\alpha,\\
		dV_{HT} & =h(b) dV_{BH}= \frac{(2+b^2)(2-3b^2)}{4(1-b^2)\sqrt{1-b^2}} \cdot dV_{BH},
		\end{split}
		\end{equation*}
		so  the meaning of the function $h$ in  Lemma 
		\ref{lem: monoton f,g,h} is clear now. 
		
		By taking into account the monotonicity of $f$, $g$, $h$ described in Lemma 
		\ref{lem: monoton f,g,h}, the inequalities stated above hold good.
		
		$\qedd$
	\end{proof}

	Moreover, from Lemma 
	\ref{lem: monoton f,g,h}  we have 
	\begin{theorem}\label{thm: area comparison2}
	Let $(M,F)$ be a slope metric on a surface of revolution. Then
\begin{enumerate}
\item $\frac{8}{9}Area_\alpha (D)\leq Area_{BH}(D)\leq Area_\alpha (D)$,
\item $\frac{5\sqrt{3}}{9}Area_\alpha (D)\leq Area_{HT}(D)\leq Area_\alpha (D)$,
\item $Area_{BH} (D)\leq Area_{HT}(D)\leq {\frac{5\sqrt{3}}{8}}Area_{BH} (D)$,
\end{enumerate}
	for any bounded region $D\subset M$.
\end{theorem}


\end{document}